\documentclass[letterpaper, 10 pt, conference]{ieeeconf}  % Comment this line out if you need a4paper

\IEEEoverridecommandlockouts                              % This command is only needed if
\usepackage{graphicx}
\usepackage{amsmath}
\usepackage{amssymb}
\usepackage{comment}
\usepackage{color}
\usepackage{epstopdf}
\usepackage{float}
\usepackage{mathrsfs}
\usepackage{hyperref}
\makeatletter
\let\theoremstyle\@undefined                        % undefine \theoremstyle
\makeatother

\usepackage{amsthm}

\newtheorem{nntheorem}{\bf Theorem}
\newtheorem{nnassumption}{\bf Assumption}
\newtheorem{nndefinition}{\bf Definition}
\newtheorem{nnlemma}{\bf Lemma}
\newtheorem{nncorollary}{\bf Corollary}
\newtheorem{nnproposition}{\bf Proposition}
\newtheorem{nexample}{\bf Example}

\newenvironment{theorem}
{\begin{nntheorem}\it}
{\end{nntheorem}}

\newenvironment{nnexample}
{\begin{nexample}\rm}{\end{nexample}}

\newtheorem{nnremark}{\bf Remark}
\newenvironment{remark}{\begin{nnremark}}{\hfill \hspace*{1pt}\hfill $\circ$\end{nnremark}}

\def\epsilon{\varepsilon}

\title{\LARGE \bf
Output Feedback Control of the Linear Korteweg-de Vries Equation*}

\author{Swann Marx$^{1}$ and Eduardo Cerpa$^{2}$% <-this % stops a space
\thanks{*This work has been partially supported by FONDECYT grant 1140741, MathAmSud COSIP, and CONICYT grant ACT-1106}% <-this % stops a space
\thanks{$^{1}$Swann Marx is a student from the Electrical Engineering Department, \'Ecole Normale Sup\'erieure de Cachan, 61 Avenue du Pr\'esident Wilson, 94230 Cachan, France
        {\tt\small swann.marx@ens-cachan.fr}.}%
\thanks{$^{2}$Eduardo Cerpa is with  Departamento de Matem\'atica, Universidad T\'ecnica Federico
Santa Mar\'ia, Avda. Espa\~na 1680, Valpara\'iso, Chile 
{\tt\small eduardo.cerpa@usm.cl}.}%
}

\begin{document}

\maketitle
%\thispagestyle{empty}
%\pagestyle{empty}
%%%%%%%%%%%%%%%%%%%%%%%%%%%%%%%%%%%%%%%%%%%%%%%%%%%%%%%%%%%%%%%%%%%%%%%%%%%%%%%%
\begin{abstract}
This paper presents the design of an output feedback control for a linear Korteweg-de Vries equation. This design is based on the backstepping method which uses a Volterra transformation. An appropriate observer is introduced and the exponential stability of the closed-loop system is proven.      
\end{abstract}

%%%%%%%%%%%%%%%%%%%%%%%%%%%%%%%%%%%%%%%%%%%%%%%%%%%%%%%%%%%%%%%%%%%%%%%%%%%%%%%


\section{Introduction}

The Korteweg-de Vries equation, introduced in 1895 by the Dutch mathematicians Diederik J. Korteweg and his student Gustav de Vries, describes approximatively the behavior of long waves in a water channel of relatively shallow depth. This nonlinear partial differential equation, described by
\begin{equation}
u_t(t,x)+u_x(t,x)+u_{xxx}(t,x)+u(t,x)u_x(t,x)=0,
\end{equation}
has been deeply studied in the controllability sense (see \cite{rosier-zhang, cerpa-rivas-zhang, cerpa2013control} and the references therein). By considering di\-ffe\-rent boundary conditions on an interval $[0,L]$ and different boundary actuators, we get control results of different nature. Roughly speaking, the system is exactly controllable when the control acts from the right endpoint $x=L$, and null-controllable when the control acts from the left endpoint $x=0$.

In this article we focus, as a first step opening further research, on the stabilizability problem for the linear Korteweg-de Vries equation with a control acting on the left Dirichlet boundary condition. 
%\textcolor{red}{The nonlinear part of this problem is under investigation, but it needs numerous computations we cannot include here due to space limitations}. 
The studied system can be written as follows:
\begin{equation}
\label{lKdV_equation_ld}
\left\{
\begin{split}
&u_t+u_{x}+u_{xxx}=0,\\
&u(t,0)=\kappa(t),\: u(t,L)=0,\: u_x(t,L)=0,\\
&u(0,x)=u_0(x),
\end{split}
\right.
\end{equation}
where $\kappa$ denotes the control input and $u_0$ is the initial condition.

Some full state feedback controls have already been designed in the literature. Let us mention \cite{cerpa2009rapid} where a Gramian-based method is applied in the case where the control acts on the right endpoint and \cite{tang2013stabKdV} where the backstepping method is applied with different boundary conditions. With the same approach than the last paper, we have \cite{cerpa_coron_backstepping} where system \eqref{lKdV_equation_ld} is considered.

However, in most cases, we have no access to the full state of the system, and it is more realistic to design an output feedback control, i.e.,  a feedback law depending only on some partial measurements of the state. 

For linear and autonomous finite-dimension systems, a stabilizability and an observability assumptions are sufficient to ensure that the separation principle holds. In other words, if there exists a controller, which asymptotically stabilizes  the origin of the system and an observer which converges asymptotically to the state system, the output feedback built from this observer and this state feedback asymptotically stabilizes  the origin of the system. In a PDE framework this principle is no longer true and the stability of the closed-loop system is not guaranteed.

The basic question to state the problem is which kind of measurements are we going to consider. The case of a boundary measurement is the most challenging one and the natural choice for the KdV equation \eqref{lKdV_equation_ld} should be $y(t)=u_x(t,0)$. Unfortunately, the system is not observable with this choice. In fact, if $L=2\pi$, then the stationary solution $u(t,x)=1-\cos(x)$, satisfies \eqref{lKdV_equation_ld} with $\kappa(t)=0$ and in addition $u_x(t,2\pi)=0$ for any time $t$. The length $2\pi$ is not the only one for which observability does not hold (see e.g. \cite{rosier1997kdv,cerpa2013control}).

In this paper we consider the output given by 
\begin{equation}
\label{measurement}
y(t)=u_{xx}(t,L).
\end{equation}
By using this measurement, we build an observer and apply the backstepping method to design an output feedback control which exponentially stabilizes the closed-loop system. The trace $u_{xx}(t,L)$ can be considered itself as a boundary condition or even a control (\cite{cg,cerpa-rivas-zhang}) and therefore the choice of this output is not an artificial one. 

%%%%%%%%%%%%%%

This paper is organized as follows. In Section \ref{main_results}, we state our main result. Section \ref{state_feedback_design} is devoted to recall the state feedback control designed in \cite{cerpa_coron_backstepping}. In \ref{regularities_results}, we state some regularity results needed to consider the measurement \eqref{measurement} as a continuous function. The observer is built in Section \ref{observer_design}. In Section \ref{stability_analysis_output_feedback}, the stability of the closed loop controller-observer system is proven. Finally, Section \ref{con} states some conclusions.

\section{Main Result}
\label{main_results}

Based on \cite{krstic_smyshlyaev_backstepping} and \cite{smyshlyaev2005backstepping}, we construct the following observer:
\begin{equation}
\label{linear_KdV_observer}
\left\{
\begin{split}
&\hat{u}_t+\hat{u}_x+\hat{u}_{xxx}+p_1(x)[y(t)-\hat{u}_{xx}(t,L)]=0,\\
&\hat{u}(t,0)=\kappa(t),\: \hat{u}(t,L)=\hat{u}_x(t,L)=0,\\
&\hat{u}(0,x)=0.
\end{split}
\right.
\end{equation}

\begin{nntheorem}\label{main-th} For any $\lambda>0$, there exist a feedback law $\kappa(t):=\kappa(\hat{u}(t,x))$, a function $p_1=p_1(x)$, and a constant $C>0$ such that the coupled system \eqref{lKdV_equation_ld}-\eqref{measurement}-\eqref{linear_KdV_observer}
%\begin{equation}
%\left\{
%\begin{split}
%&u_t+u_{x}+u_{xxx}=0,\\
%&u(t,0)={\kappa}(t),\: u(,L)=u_x(t,L)=0,\\
%&\hat{u}_t+\hat{u}_x+\hat{u}_{xxx}+p_1(x)[u_{xx}(t,L)-\hat{u}_{xx}(t,L)]=0,\\
%&\hat{u}(t,0)=\kappa(t),\: \hat{u}(t,L)=\hat{u}_x(t,L)=0,\\
%\end{split}
%\right.
%\end{equation}
is globally exponentially stable with a decay rate equals to $\lambda$, i.e., for any $u_0\in H^3(0,L)$ we have
\begin{equation}
\begin{split}
\Vert u(t,\cdot)\Vert_{H^3(0,L)}+&\Vert \hat{u}(t,\cdot)\Vert_{L^2(0,L)}\leq
C e^{-\lambda t}\Vert u_0\Vert_{H^3(0,L)}
\end{split}
\end{equation}

\end{nntheorem}

\section{Full state feedback design}
\label{state_feedback_design}

In \cite{cerpa_coron_backstepping} the following Volterra transformation is introduced
\begin{equation}\label{mapk}
w(x)=\Pi(u(x)):=u(x)-\int_x^Lk(x,y)u(y)dy.
\end{equation}
The function $k$ is chosen such that the trajectory $u=u(t,x)$, solution of (\ref{lKdV_equation_ld}), with control
\begin{equation}\label{feedback} \kappa(t)=\int_0^Lk(0,y)u(t,y)dy,
\end{equation}
is  mapped into the trajectory $w=w(t,x)$, solution of the linear system
\begin{equation}
\label{sf_voltransf}
\left\{
\begin{split}
&w_t+w_x+w_{xxx}+\lambda w=0,\\
&w(t,0)=w(t,L)=w_x(t,L)=0,
\end{split}
\right.
\end{equation}
which is exponentially stable if $\lambda>0$. In fact, from a Lyapunov approach, we get
\begin{equation}\label{calcul}
\begin{split}
\frac{d}{dt}\int_0^L |w(t,x)|^2dx &=-|w_x(t,0)|^2-2\lambda\int_0^L |w(t,x)|^2dx\\
&\leq -2\lambda\int_0^L |w(t,x)|^2dx,
\end{split}
\end{equation}
which gives an exponential decay rate equals to $\lambda$ for the $L^2-$norm of the state $w$.

The kernel function $k=k(x,y)$ is characterized by:
\begin{equation}
\label{sf_kernel}
\left\{
\begin{split}
&k_{xxx}+k_{yyy}+k_x+k_y=-\lambda k,\text{ in }\mathcal{T},\\
&k(x,L)=0,\text{ in } [0,L],\\
&k(x,x)=0,\text{ in } [0,L],\\
&k_x(x,x)=\frac{\lambda}{3}(L-x),\text{ in }[0,L],
\end{split}
\right.
\end{equation}
where $\mathcal{T}:=\lbrace (x,y)/x\in [0,L],y\in [x,L]\rbrace$. The solution of (\ref{sf_kernel}) exist. This is proved by using the method of successive approximations. Unlikely the case of heat or wave equations, we do not have an explicit solution.

In \cite{cerpa_coron_backstepping} it is proved that the transformation \eqref{mapk} linking (\ref{lKdV_equation_ld}) and (\ref{sf_voltransf}) is invertible, continuous and its inverse is also continuous. Therefore, the exponential decay for $w$, solution of \eqref{sf_voltransf}, implies the exponential decay for the solution $u$ controlled by \eqref{feedback}. Thus, with this method, the following theorem is proven.

\begin{theorem} (State feedback stabilization for KdV (\cite{cerpa_coron_backstepping})

For any $\lambda>0$, there exist a feedback control law $\kappa=\kappa(u(t,.))$ and $C>0$ such that
\begin{equation}
\Vert u(t,.)\Vert_{L^2(0,L)}\leq Ce^{-\lambda t}\Vert u_0\Vert_{L^2(0,L)}
\end{equation}
for any solution of \eqref{lKdV_equation_ld}-\eqref{feedback}.
\end{theorem}

\section{Regularity Result}
\label{regularities_results}

As we said in the introduction, we consider
\begin{equation}
y(t)=u_{xx}(t,L)
\end{equation}
 as a partial measurement of the solution. However, since we have the trace of the second derivative with respect to $x$ of $u$, we need a regularity stronger than in \cite{cerpa_coron_backstepping}. Indeed, we ask the output $y(t)$ to be a continuous function. Thus we have the following lemma.

\begin{nnlemma} \label{lemma_wp} Let us consider system 
\begin{equation}
\left\{
\begin{split}
&u_t+u_{x}+u_{xxx}=0,\\
&u(t,0)=\kappa(t),\: u(t,L)=0,\: u_x(t,L)=0,\\
&u(0,x)=u_0(x),
\end{split}
\right.
\end{equation}
where $u_0\in H^3(0,L)$ and $\kappa(t)\in H^1(0,T)$. Then $u\in C([0,T],H^3(0,L))\cap L^2(0,T;H^4(0,L))$
and $u_{xx}(\cdot,L)\in C([0,T])$.
\end{nnlemma}

\begin{proof}
This proof is based on \cite{glass2008some}. Let us consider the following coordinates transformation
\begin{equation}
v=u_t.
\end{equation}
The dynamics of $v$ can be written as follows:
\begin{equation}
\label{observability_new_coordinates}
\left\{
\begin{split}
&v_t+v_{xxx}+v_x=0,\\
&v(t,0)=\dot{\kappa}(t)\in L^2(0,T),\: v(t,L)=v_x(t,L)=0,\\
&v(0,x)=(-u_0^{\prime\prime\prime}-u_0^\prime)\in L^{2}(0,L).
\end{split}
\right.
\end{equation}

By already known well-posedness results for KdV (\cite{glass2008some}), we get 
\begin{equation}
\begin{split}
&v\in C([0,T],L^{2}(0,L))\cap L^2(0,T;H^1(0,L))\\
\Rightarrow &u_t\in C([0,T],L^{2}(0,L))\cap L^2(0,T;H^1(0,L))\\
\Rightarrow &u\in C([0,T],H^3(0,L))\cap L^2(0,T;H^4(0,L))\\
\Rightarrow &u_{xx}\in C([0,T],H^1(0,L))\cap L^2(0,T;H^2(0,L))\\
\Rightarrow &u_{xx}\in C([0,T]\times[0,L])\\
\Rightarrow &u_{xx}(\cdot,L)\in C([0,T])\\
\end{split}
\end{equation}

Thus it concludes the proof of Lemma \ref{lemma_wp}.
\end{proof}

\section{Observer design}
\label{observer_design}

Based on \cite{krstic_smyshlyaev_backstepping}, and more precisely on \cite{smyshlyaev2005backstepping}, we can write, for system
\begin{equation}
\label{lKdV_observer}
\left\{
\begin{split}
&u_t+u_x+u_{xxx}=0,\\
&u(t,0)=\kappa(t),\: u(t,L)=u_x(t,L)=0,\\
&y(t)=u_{xx}(t,L),
\end{split}
\right.
\end{equation}
the corresponding observer
\begin{equation}
\label{linear_KdV_observer2}
\left\{
\begin{split}
&\hat{u}_t+\hat{u}_x+\hat{u}_{xxx}+p_1(x)[u_{xx}(t,L)-\hat{u}_{xx}(t,L)]=0,\\
&\hat{u}(t,0)=\kappa(t),\: \hat{u}(t,L)=\hat{u}_x(t,L)=0.
\end{split}
\right.
\end{equation}

The construction of the observer is based on the finite-dimensional design for $\dot{x}=Ax+Bu$,\: $y=Cx$, 
which proposes the observer $\dot{\hat{x}}=A\hat{x}+Bu+L(y-C\hat{x})$. If we consider 
 the error $e=x-\hat{x}$, then $\dot{e}=(A-LC)e$, and we have to look for a matrix $L$ insuring a good performance. Because of the infinite-dimensional framework we are working in, a matrix $L$ is not enough. Thus we need a function $p_1(x)$.

In our case, we consider the error $\tilde{u}:=u-\hat{u}$, which satisfies
\begin{equation}
\label{linear_KdV_error}
\left\{
\begin{split}
&\tilde{u}_t+\tilde{u}_x+\tilde{u}_{xxx}-p_1(x)\tilde{u}_{xx}(t,L)=0,\\
&\tilde{u}(t,0)= \tilde{u}(t,L)=\tilde{u}_x(t,L)=0.\\
\end{split}
\right.
\end{equation}

Given a positive parameter $\lambda$, we look for a transformation $\Pi_o$ defined by
\begin{equation}
\label{transformation}
\tilde{u}(x)=\Pi_o(\tilde{w}(x))=\tilde{w}(x)-\int_x^L p(x,y)\tilde{w}(y)dy
\end{equation}
such that the trajectory $\tilde{u}$, solution of (\ref{linear_KdV_error}) is mapped into the trajectory $\tilde w=\tilde w(t,x)$, solution of the linear system
\begin{equation}
\label{linear_KdV_transformed}
\left\{
\begin{split}
&\tilde w_t+\tilde w_x+\tilde w_{xxx}+\lambda \tilde w=0,\\
&\tilde w(t,0)=0,\: \tilde w(t,L)=0,\: \tilde w_x(t,L)=0,
\end{split}
\right.
\end{equation}
which is exponentially stable with a decay rate depending on the value of $\lambda$ as shown in \eqref{calcul}.

Now, the key step is to find the kernel $p=p(x,y)$ such that $\tilde{u}(t,x)=\Pi_o(\tilde w(t,x))$ satisfies (\ref{linear_KdV_error}). By focusing on (\ref{transformation}) and using the Leibniz rules, we get:

$\bullet$ Differentiation along (\ref{linear_KdV_transformed})
\begin{equation}
\label{observer_difftime}
\begin{split}
\tilde{u}_t &=\tilde{w}_t(x)-\int_x^L p(x,y)[-\tilde{w}_y(y,t)-\tilde{w}_{yyy}(y,t)-\lambda\tilde{w}(y,t)]dy\\
&=\tilde{w}_t(x)-\int_x^L (-\lambda p(x,y) +p_y(x,y)+p_{yyy}(x,y))\tilde{w}(y,t)dy\\
&+p(x,L)\tilde{w}(L,t)-p(x,x)\tilde{w}(x,t)+p(x,L)\tilde{w
}_{xx}(L,t)\\ 
&-p(x,x)\tilde{w}_{xx}(x,t)+p_{y}(x,x)\tilde{w}_{x}(x)-p_{y}(x,L)\tilde{w}_{x}(L,t)\\
&+p_{yy}(x,L)\tilde{w}(L,t)-p_{yy}(x,x)\tilde{w}(x,t)
\end{split}
\end{equation}

$\bullet$ Three differentations with respect to the variable $x$
\begin{equation}
\label{observer_diffx1}
\begin{split}
\tilde{u}_x(x,t)=\tilde{w}_x(x,t)+p(x,x)\tilde{w}(x,t)-\int_x^L p_x(x,y)\tilde{w}(t,y)dy
\end{split}
\end{equation}

\begin{equation}
\label{observer_diffx2}
\begin{split}
\tilde{u}_{xx}(t,x)&=\tilde{w}_{xx}(t,x)+\frac{d}{dx}p(x,x)\tilde{w}(t,x)+p(x,x)\tilde{w}_x(t,x)\\
&+p_x(x,x)\tilde{w}(t,x)-\int_x^L p_{xx}(x,y)\tilde{w}(t,y)dy
\end{split}
\end{equation}

\begin{equation}
\label{observer_diffx3}
\begin{split}
\tilde{u}_{xxx}(t,x) &=\tilde{w}_{xxx}(t,x)+\frac{d^2}{dx^2}p(x,x)\tilde{w}(t,x)\\
&+2\frac{d}{dx}p(x,x)\tilde{w}_x(t,x)\\
&+p(x,x)\tilde{w}_{xx}(t,x)+\frac{d}{dx}p_x(x,x)\tilde{w}(t,x)\\
&+p_x(x,x)\tilde{w}_x(t,x)+p_{xx}(x,x)\tilde{w}(t,x)\\
&-\int_x^L p_{xxx}(x,y)\tilde{w}(t,x)dy
\end{split}
\end{equation}

By adding (\ref{observer_difftime}), (\ref{observer_diffx1}) and (\ref{observer_diffx3}), we get
\begin{equation}
\label{link_transformation}
\begin{split}
&\tilde{u}_t+\tilde{u}_x+\tilde{u}_{xxx}-p_1(x)\tilde{u}_{xx}(L)=\\
&\tilde{w}_t(t,x)+\tilde{w}_{x}(t,x)+\tilde{w}_{xxx}(t,x)+\lambda\tilde{w}(t,x)\\
&-\int_x^L (-\lambda p(x,y) +p_y(x,y)+p_{yyy}(x,y)\\
&+p_{xxx}(x,y)+p_x(x,y))\tilde{w}(y,t)dy\\
&+\tilde{w}_x(t,x)\left(2\frac{d}{dx}p(x,x)+p_x(x,x)+p_y(x,x)\right)\\
&+\tilde{w}(t,x)\displaystyle\left(p_{xx}(x,x)+\frac{d^2}{dx^2}p(x,x)+\frac{d}{dx}p_x(x,x)\right.\\
&\null\hfill\displaystyle\left. -p_{yy}(x,x)-\lambda\right)+p(x,L)\tilde{w}(L,t)+(p(x,L)-p_1(x))\tilde{w}_{xx}(L,t)\\
&-p_y(x,L)\tilde{w}_{x}(L,t).
\end{split}
\end{equation}

From this equation, we get four conditions:

\begin{itemize}
\item[1.] Equation for $(x,y)\in \mathcal{T}$:
\begin{equation}
\label{1.observer}
\begin{split}
&p_{yyy}(x,y)+p_{xxx}(x,y)
+p_y(x,y)+p_x(x,y)=\lambda p(x,y).
\end{split}
\end{equation}
\item[2.] First boundary condition on $(x,x)$ for $x\in [0,L]$:
\begin{equation}
\label{2.observer}
2\frac{d}{dx}p(x,x)+p_x(x,x)+p_y(x,x)=0.
\end{equation}
\item[3.] Second boundary condition on $(x,x)$ for $x\in [0,L]$:
\begin{equation}
\label{3.observer}
\begin{split}
&\frac{d^2}{dx^2}p(x,x)+\frac{d}{dx}p_x(x,x)\\
&+p_{xx}(x,x)-p_{yy}(x,x)-\lambda=0.
\end{split}
\end{equation}
\item[4.] Appropriate choice of $p_1$:
\begin{equation}
\label{4.observer}
p(x,L)=p_1(x).
\end{equation}  
\end{itemize}
Recall that $\mathcal{T}:=\lbrace (x,y)/x\in [0,L],\: y\in [x,L]\rbrace$.

Moreover, note also that, by setting $x=0$ in (\ref{transformation}), we get:

\begin{equation}
\label{5.observer}
p(0,y)=0,\hspace{0.2cm}\forall y\in[0,L].
\end{equation}

In addition, we have

$$\tilde{w}(t,L)=\tilde{u}(t,L)=\tilde{w}_x(t,L)=0.$$

Finally, the kernel $p$ satisfies the following PDE:
\begin{equation}
\left\{
\begin{split}
&p_{xxx}(x,y)+p_{yyy}(x,y)\\
&+p_{y}(x,y)+p_x(x,y)=\lambda p(x,y),\hspace{0.3cm}(x,y)\in\mathcal{T},\\
&p(x,x)=0,\hspace{0.3cm}x\in [0,L],\\
&p_x(x,x)=\frac{\lambda}{3}x,\hspace{0.3cm}x\in [0,L],\\
&p(0,y)=0,\hspace{0.3cm}y\in [0,L].
\end{split}
\right.
\end{equation}

Let us make the following change of variable:

\begin{equation}
\bar{x}=L-y,\hspace{0.5cm}\bar{y}=L-x,
\end{equation}
and define $F(\bar{x},\bar{y}):=p(x,y)$. Hence:
\begin{equation}
\left\{
\begin{split}
&F_{\bar{x}\bar{x}\bar{x}}(\bar{x},\bar{y})+F_{\bar{y}\bar{y}\bar{y}}(\bar{x},\bar{y})\\
&+F_{\bar{y}}(\bar{x},\bar{y})+F_{\bar{x}}(\bar{x},\bar{y})=-\lambda F(\bar{x},\bar{y})\hspace{0.3cm}(\bar{x},\bar{y})\in\mathcal{T}\\
&F(\bar{x},\bar{x})=0\hspace{0.3cm}\bar{x}\in [0,L]\\
&F_{\bar{x}}(\bar{x},\bar{x})=\frac{\lambda}{3}(L-\bar{x})\hspace{0.3cm}\bar{x}\in [0,L]\\
&F(\bar{x},L)=0\hspace{0.3cm}\bar{y}\in [0,L]
\end{split}
\right.
\end{equation}

This PDE has already been studied in \cite{cerpa_coron_backstepping}, where no explicit solution has been found, but where the existence of a solution has been proved. Hence, we can conclude that the kernel $p:=p(x,y)$ exists. Note that the function $\Pi_o$ defined by (\ref{transformation}) is linear (by definition) and continuous (because of the existence of $p$). 

\section{Stability analysis of the closed loop system}
\label{stability_analysis_output_feedback}

Instead of dealing directly with the controlled state $u$ and the observer state $\hat u$, we consider the evolution of the couple $(\tilde u, \hat u)$ where $\tilde u$ stands for the error $\tilde u=u-\hat u$, as introduced in Section \ref{observer_design}.

By using the output feedback control
\begin{equation}\label{feedback2} \kappa(t)=\int_0^L k(0,y)\hat u(t,y)dy,
\end{equation}
the transformation $\Pi$ defined in \eqref{mapk} and its inverse, and the transformation  $\Pi_o$ defined in \eqref{transformation} and its inverse, we can see that $(\tilde u, \hat u)$ are mapped into $(\tilde w, \hat w)=(\Pi_o^{-1}(\tilde u),\Pi(\hat u))$ solutions of the target system
\begin{equation}\label{target3}
\left\{
\begin{split}
&\hat{w}_t+\hat{w}_{x}+\hat{w}_{xxx}+\lambda\hat{w}=\\
&-\left\{p_1(x)-\int_x^L k(x,y)p_1(y)dy\right\}\tilde{w}_{xx}(t,L),\\
&\hat{w}(0)=\hat{w}(L)=\hat{w}_x(L)=0,\\
&\tilde{w}_t+\tilde{w}_x+\tilde{w}_{xxx}+\lambda\tilde{w}=0,\\
&\tilde{w}_x(0)=\tilde{w}(L)=\tilde{w}_x(L)=0.
\end{split}
\right.
\end{equation}

%Consider the controller computed in \cite{cerpa_coron_backstepping} and keep the same notations. It is straightforward to show that the observer and control backstepping transformations  and:
%\begin{equation}
%\hat{w}(x)=u(\hat{x})-\int_x^L k(x,y)\hat{u}(y)dy
%\end{equation}
%\begin{equation}
%\hat{u}(x)=w(\hat{x})+\int_x^L l(x,y)\hat{w}(y)dy
%\end{equation}
%map the closed-loop system consisting of the observer error PDE and the observer into the following target system:
%\begin{equation}
%\left\{
%\begin{split}
%&\hat{w}_t+\hat{w}_{x}+\hat{w}_{xxx}+\lambda\hat{w}=\\
%&-\left\{p_1(x)-\int_x^L k(x,y)p_1(y)dy\right\}\tilde{w}_{xx}(t,L)\\
%&\hat{w}(0)=\hat{w}(L)=\hat{w}_x(L)=0\\
%&\tilde{w}_t+\tilde{w}_x+\tilde{w}_{xxx}+\lambda\tilde{w}=0\\
%&\tilde{w}_x(0)=\tilde{w}(L)=\tilde{w}_x(L)=0
%\end{split}
%\right.
%\end{equation}
Note that the parameter $\lambda$ is the same for the observer and the system itself.

Given that $\Pi,\Pi_o$ are continuous maps, invertible and their inverse maps are also continuous, the exponential stability of \eqref{target3} would imply the exponential stability of the closed loop system and therefore the proof of  Theorem \ref{main-th} will be ended.

In order to prove the exponential stability of \eqref{target3}, we use a Lyapunov argument. Let us consider the following function,

\begin{equation} V(t)=V_1(t)+V_2(t)+V_3(t),
\end{equation}
where \begin{equation} V_1(t)=\frac A 2 \int_0^L|\hat w(t,x)|^2\,dx
\end{equation}
 \begin{equation} V_2(t)=\frac B 2 \int_0^L|\tilde w(t,x)|^2\,dx
\end{equation}
 \begin{equation} V_3(t)=\frac B 2 \int_0^L|\tilde w_{t}(t,x)|^2\,dx
\end{equation}
with $A,B$ to be chosen later.

\begin{remark}\label{r1}
We can prove that this Lyapunov function is equivalent to the one obtained by replacing $V_3(t)$ by 
\begin{equation} \tilde V_3 (t)=\frac B 2 \int_0^L|\tilde w_{xxx}(t,x)|^2\,dx\end{equation}
By this, we mean that the exponential decay of one of them implies the exponential decay of the other one. In fact, we can prove that there exist positive constants $d_1,d_2$ such that
$$d_1(V_2(t)+\tilde V_3(t))\leq V_2(t)+V_3(t)\leq d_2(V_2(t)+\tilde V_3(t)).$$
\end{remark}

Taking the time derivative of the function $V(t)$, we get after some computations that
\begin{equation*}
\begin{split}
\dot V_1(t)=&A\int_0^L\hat{w}_t(t,x)\hat{w}(t,x)dx\\
&\leq (-A\lambda +D^2)\int_0^L |\hat{w}(t,x)|^2 dx + A^2|\tilde{w}_{xx}(t,L)|^2\\
&= 2\Big(-\lambda +\frac{D^2}{A} \Big)V_1(t) + A^2|\tilde{w}_{xx}(t,L)|^2 
%-A\lambda\int_0^L |\hat{w}(t,x)|^2dx &-A|\tilde{w}_{xx}(t,L)|.\\
%&\int_0^L|\hat{w}(t,x)|\left|\left\{p_1(x)-\int_x^L k(x,y)p_1(y)dy\right\}\right|  dx\\
%\leq &-A\lambda\int_0^L|\hat{w}(t,x)|^2dx -B\lambda\Vert \tilde{w}(t,x)\Vert_{H^3(0,L)}\\
%&+A^2 |\tilde{w}_{xx}(t,L)|^2\\
%&+C^2\int_0^L|\hat{w}(t,x)|^2  dx\\
%&\hspace{0.3cm}\text{where $C:=\max_{x\in [0,L]}\left\{p_1(x)-\int_x^L k(x,y)p_1(y)dy\right\}$}\\
%\leq & (-A\lambda +D^2)\int_0^L |\hat{w}(t,x)|^2 dx -B\lambda\Vert \tilde{w}(t,x)\Vert_{H^3(0,L)}\\
%&+A^2|\tilde{w}_{xx}(t,L)|^2 
\end{split}
\end{equation*}
where $D:=\max_{x\in [0,L]}\left\{p_1(x)-\int_x^L k(x,y)p_1(y)dy\right\}$.

From the same computation as in \eqref{calcul}, we see that 
\begin{equation*}
\begin{split}
\dot V_2(t)\leq&-2\lambda V_2(t).
\end{split}
\end{equation*}

Moreover, thanks to the regularity $H^3(0,L)$, the same computation can be applied to $z=\tilde w_t$ (see the proof of Lemma \ref{lemma_wp}) to obtain
\begin{equation*}
\begin{split}
\dot V_3(t)\leq&-2\lambda V_3(t)
\end{split}
\end{equation*}
%B\int_0^L (\tilde{w}_{t}(t,x)+\tilde{w}_{xt}(t,x)+\\
% &\tilde{w}_{xxt}(t,x)+\tilde{w}_{xxxt}(t,x))(\tilde{w}(t,x)+\tilde{w}_x+\tilde{w}_{xx}(t,x)\\
%&+\tilde{w}_{xxx}(t,x))dx\\
%-B\lambda\Vert \tilde{w}(t,x)\Vert_{H^3(0,L)}\\
%&+

%-B\lambda\Vert \tilde{w}(t,x)\Vert_{H^3(0,L)}^2\\

Thus, we get:
\begin{equation*}
\begin{split}
\dot V(t)\leq&2\Big(-\lambda +\frac{D^2}{A} \Big)V_1(t)+ A^2|\tilde{w}_{xx}(t,L)|^2\\
&-2\lambda V_2(t)-2\lambda V_3(t).
\end{split}
\end{equation*}

%\begin{equation}
%\begin{split}
%\lim_{t\rightarrow +\infty}&(\Vert \hat{w}(t,x)\Vert_{L^2(0,L)}\\
%&+\Vert \tilde{w}(t,x)\Vert_{H^3(0,L)})\leq\lim_{t\rightarrow +\infty}|\tilde{w}_{xx}(t,L)|^2
%\end{split}
%\end{equation}
%This inequality reminds us of the notion of input-output-to-state stability (IOSS), which is well explained in \cite{iss_sontag}. This notion is a robustness result, which permits to find a stability result whatever the control and the output are. One can find some results related to the ISS notion in \cite{prieur2011iss} and \cite{mazenc2011strict}.

We need to find an upper bound for $|\tilde{w}_{xx}(t,L)|^2$. We multiply
\begin{equation}
\label{calcul_ISS_KdV}
\left\{
\begin{split}
&\tilde{w}_t+\tilde{w}_x+\tilde{w}_{xxx}+\lambda \tilde{w}=0,\\
&\tilde{w}(0)=\tilde{w}(L)=\tilde{w}_x(L)=0,
\end{split}
\right.
\end{equation}
%Let multiply the first line of \eqref{calcul_ISS_KdV} by 
%$x^2w_x$ and then integrate the result from $0$ to $L$ with respect to $x$. Then we get that:
%
%\begin{equation}
%\label{IPP_ISS1}
%\Vert \tilde{w}_{xx}\Vert_{L^2(0,L)}\leq \left(\frac{L^2}{2}+\frac{1}{L}\right)\Vert \tilde{w}_x\Vert_{L^2(0,L)}+\frac{L^2}{2}\Vert \tilde{w}_t\Vert_{L^2(0,L)}
%\end{equation} 
%
%We get also by multiplying by $x^2\tilde{w}$:
%
%\begin{equation}
%\label{IPP_ISS2}
%\Vert \tilde{w}_x\Vert_{L^2(0,L)}\leq \frac{2L+L^3}{6}\Vert \tilde{w}\Vert_{L^2(0,L)}+\frac{1}{6L}\Vert \tilde{w}_t\Vert_{L^2(0,L)}
%\end{equation}
by $x\tilde{w}_{xx}$ and after some computations we get

\begin{equation}
\begin{split}
\label{IPP_ISS3}
|\tilde{w}_{xx}(t,L)|^2\leq &\left(\frac{1}{L}+L\right)\Vert \tilde{w}_{xx}\Vert_{L^2(0,L)}^2 \\
&+\left(2\lambda +\frac{1}{L}\right)\Vert \tilde w_x\Vert_{L^2(0,L)}^2 \\
& +\frac{1}{L}\Vert \tilde{w}_t\Vert_{L^2(0,L)}^2
\end{split}
\end{equation}
and finally the existence of $a,b>0$ such that
\begin{equation}
\label{ISS_final}
|\tilde{w}_{xx}(t,L)|^2\leq a\Vert \tilde{w}\Vert^2_{L^2(0,L)} +b\Vert\tilde{w}_t\Vert^2_{L^2(0,L)}
\end{equation}

\begin{remark} Here, we have used that the norm $\|f\|_{H^3(0,L)}$ and the norm
$\|f\|_{L^2(0,L)}+\|f_{xxx}\|_{L^2(0,L)}$ are equivalent. See also Remark \ref{r1}.
\end{remark}

We use the latter inequality to write:
\begin{equation*}
\begin{split}
\dot V(t)\leq&2\Big(-\lambda +\frac{D^2}{A} \Big)V_1(t)+ 2a \frac{A^2}{B} V_2(t)\\
&+2b \frac{A^2}{B} V_3(t)-2\lambda V_2(t)-2\lambda V_3(t).
\end{split}
\end{equation*}
Therefore, 
\begin{equation*}
\begin{split}
\dot V(t)\leq&2\Big(-\lambda +\frac{D^2}{A} \Big)V_1(t)+ 2\Big(-\lambda +\frac{aA^2}{B} \Big)V_2(t)\\
&2\Big(-\lambda +\frac{bA^2}{B} \Big)V_3(t).
\end{split}
\end{equation*}

In this way, by tuning $A,B$ large enough, we get for any $\epsilon>0$ that
\begin{equation*}
\begin{split}
\dot V(t)\leq 2\Big(-\lambda +\epsilon \Big)V(t),
\end{split}
\end{equation*}
which gives an exponential stability with decay rate as close to $\lambda$ as we want. The rapid stabilization is achieved because the parameter $\lambda$ can be chosen as large as desired.

It concludes the proof of the stability of the closed loop system with the output feedback control law depending on a boundary measurement of the state.

\section{Conclusions}\label{con}

In this paper we have designed an output feedback controller for the linear Korteweg-de Vries equation posed on a bounded interval. This boundary feedback control acts on the Dirichlet boundary condition on the left endpoint and exponentially stabilizes the system by using the boundary measurement $y(t)=u_{xx}(t,L)$. Because of this choice of $y(t)$, we have to work on a more regular framework, given by $H^3(0,L)$ as the state space. The backstepping method is applied together a classical observer in order to build the control. 

This work is the first one addressing the output feedback problem for the Korteweg-de Vries equation and opens several possible extensions. Different location for the control or different boundary conditions can be considered as in \cite{cerpa_coron_backstepping, tang2013stabKdV}. The choice of the measurement is particularly interesting. For instance, it would be nice to deal with the collocated case, which could be harder than the non-collocated problem considered in this paper. Of course, the nonlinear case also appears as a natural next step in the study of this control system. In that context, even performing simulations of the closed loop system may be an interesting challenge. 
%\textcolor{red}{Doing the simulation of such a system could be a challenge: regarding the complexity of this equation, one has to choose a numerical scheme that is stable and that mixes backward and forward finite difference.}

\bibliographystyle{plain}
\bibliography{root}

\begin{thebibliography}{10}

\bibitem{cerpa2013control}
E.~Cerpa.
\newblock Control of a {K}orteweg-de {V}ries equation: a tutorial.
\newblock {\em Mathematical {C}ontrol and {R}elated {F}ields}, 4(1):45--99,
  2014.

\bibitem{cerpa_coron_backstepping}
E.~Cerpa and J.-M. Coron.
\newblock Rapid stabilization for a {K}orteweg-de {V}ries equation from the
  left {D}irichlet boundary condition.
\newblock {\em IEEE Trans. Automat. Control}, 58(7):1688--1695, 2013.

\bibitem{cerpa2009rapid}
E.~Cerpa and E.~Cr{\'e}peau.
\newblock Rapid exponential stabilization for a linear {K}orteweg-de {V}ries
  equation.
\newblock {\em Discrete Contin. Dyn. Syst. Ser. B}, 11(3):655--668, 2009.

\bibitem{cerpa-rivas-zhang}
E.~Cerpa, I.~Rivas, and B.-Y. Zhang.
\newblock Boundary controllability of the {K}orteweg-de {V}ries equation on a
  bounded domain.
\newblock {\em SIAM J. Control Optim.}, 51(4):2976--3010, 2013.

\bibitem{cg}
T.~Colin and J.-M. Ghidaglia.
\newblock An initial-boundary-value problem for the {K}orteweg-de {V}ries
  equation posed on a finite interval.
\newblock {\em Adv. Differential Equations}, 6:1463--1492, 2001.

\bibitem{glass2008some}
O.~Glass and S.~Guerrero.
\newblock Some exact controllability results for the linear kdv equation and
  uniform controllability in the zero-dispersion limit.
\newblock {\em Asymptotic Analysis}, 60(1):61--100, 2008.

\bibitem{krstic_smyshlyaev_backstepping}
M.~Krstic and A.~Smyshlyaev.
\newblock {\em Boundary {C}ontrol of {PDE}s: {A} {C}ourse on {B}ackstepping
  {D}esigns}.
\newblock SIAM, 2008.

\bibitem{rosier1997kdv}
L.~Rosier.
\newblock Exact boundary controllability for the {K}orteweg-de {V}ries equation
  on a bounded domain.
\newblock {\em ESAIM: Control, Optimisation and Calculus of Variations},
  2:33--55, 1997.

\bibitem{rosier-zhang}
L.~Rosier and B.-Y. Zhang.
\newblock Control and stabilization of the {K}orteweg-de {V}ries equation:
  recent progresses.
\newblock {\em J. Syst. Sci. Complex.}, 22(4):647--682, 2009.

\bibitem{smyshlyaev2005backstepping}
A.~Smyshlyaev and M.~Krstic.
\newblock Backstepping observers for a class of parabolic {PDE}s.
\newblock {\em Systems \& {C}ontrol letters}, 54(7):613--625, 2005.

\bibitem{tang2013stabKdV}
S.~Tang and M.~Krstic.
\newblock Stabilization of {L}inearized {K}orteweg-de {V}ries {S}ystems with
  {A}nti-diffusion.
\newblock In {\em American Control Conference (ACC)}, pages 3302--3307, 2013.

\end{thebibliography}

\end{document}